\documentclass[12pt]{article}
\usepackage[english]{babel}

\usepackage{amsmath}
\usepackage{amssymb,amsthm}

\usepackage{mathrsfs}

\usepackage{bbm}

\numberwithin{equation}{section}

\newtheorem{theorem}{Theorem}

\newtheorem{lemma}{Lemma}[section]

\newtheorem{remark}[lemma]{Remark}
\newtheorem{proposition}{Proposition}
\newtheorem{notation}[lemma]{Notation}
\newtheorem{definition}[lemma]{Definition}

\newtheorem{corollary}[lemma]{Corollary}

\renewcommand{\leq}{\leqslant}
\renewcommand{\geq}{\geqslant}

\makeatletter
\newcommand*{\IfItalic}{%
  \ifx\f@shape\my@test@it
    \expandafter\@firstoftwo
  \else
    \expandafter\@secondoftwo
  \fi
}
\newcommand*{\my@test@it}{it}
\makeatother

\newcommand{\myae}{\IfItalic{\emph{\mbox{\ae}}}{\mbox{\ae}}}

\newcommand\XLast{}
\newcommand\YLast{}

\newcommand\Vidr[4]{
\qbezier(#1,#2)(#1,#2)(#3,#4)
\renewcommand\XLast{#3}
\renewcommand\YLast{#4}}

\newcommand\VidrTo[2]{\Vidr{\XLast}{\YLast}{#1}{#2}}

\newcommand\ssemg\psi
\newcommand\fkinkg\gamma

\usepackage[all]{xy}

\usepackage{amscd}

\newcommand{\OneOp}{\mathbbm 1}

\begin{document}

\begin{center}
{\Large Piecewise linear unimodal maps\\ with non-trivial
continuous piecewise linear commutator}

{\large Makar Plakhotnyk\\
University of S\~ao Paulo, Brazil.\\

makar.plakhotnyk@gmail.com}\\
\end{center}

\begin{abstract}
Let $g:\, [0, 1]\rightarrow [0, 1]$ be piecewise linear unimodal
map. We say that $g$ has non-trivial piecewise linear commutator,
if there is a continuous  piecewise linear $\psi:\, [0,
1]\rightarrow [0, 1]$ such that $g\circ \psi = \psi\circ g$, and,
moreover, $\psi$ is neither an iteration of $g$, not a constant
map.

We prove that if $g$ has a non-trivial piecewise linear
commutator, then $g$ is topologically conjugated with the tent map
by a piecewise linear conjugacy.~\footnote{This work is partially
supported by FAPESP (S\~ao Paulo, Brazil).}~
\end{abstract}

\section{Introduction}

We will call a continuous map $g: [0, 1]\rightarrow [0, 1]$
\textbf{unimodal}, if it can be written in the form
\begin{equation}\label{eq:1.1} g(x) = \left\{
\begin{array}{ll}g_l(x),& 0\leq x\leq
v,\\
g_r(x), & v\leq x\leq 1,
\end{array}\right.
\end{equation} where %
$v\in (0,\, 1)$ is a parameter, the function $g_l$ is increasing,
the function $g_r$ is decreasing, and $$g(0)=g(1)=1-g(v)=0.$$

The fundamental example of unimodal map is the \textbf{tent} map
$$ f:\, x\mapsto 1-|1-2x|.
$$

\begin{definition}
Continuous surjective solution $\eta$ of the function equation $$
\eta \circ f = g\circ
\eta,$$ where %
$f$ is the tent map and $g$ is unimodal map, is called
\textbf{semi conjugation} from $f$ to $g$.
\end{definition}

\begin{definition}
Let $g$ be unimodal map. We will call a continuous surjective
solution $\psi$ of the functional equation
\begin{equation}\label{eq:1.2} \psi \circ g = g\circ \psi,
\end{equation} a \textbf{self semi conjugation} of
$g$.
\end{definition}

We will prove the next Theorem, which permits to reduce the number
of maximal parts of monotonicity of the map, which commutes with
the Tent map.

\begin{proposition}\label{th-5}
Let $\psi:\, [0, 1]\rightarrow [0, 1]$ be a piecewise linear
non-constant map such that~\eqref{eq:1.2} holds. If $\psi$ has
$2t$ maximal intervals of monotonicity, then there exists
non-constant piecewise linear $\widetilde{\psi}:\, [0,
1]\rightarrow [0, 1]$, which has $t$ maximal intervals of
monotonicity, and
\begin{equation}\label{eq:1.3} \widetilde{\psi}\circ g = g\circ
\widetilde{\psi}.
\end{equation}
\end{proposition}

The main result of this work is the next

\begin{theorem}\label{th:1}
Let $g$ be a unimodal map. If there exists a non-trivial piecewise
linear commutator of $g$, then $g$ is topologically conjugated
with the Tent map and, moreover, the conjugacy is piecewise
linear.
\end{theorem}

We will first prove Theorem~\ref{th:1} for the case when $g$ is
topologically conjugated with the tent map.

\begin{proposition}\label{prop:1}
Let $g$ be a unimodal map, which is topologically conjugated with
the tent map. If there exists a non-trivial piecewise linear
commutator of $g$, then the conjugacy of $g$ and the tent map is
piecewise linear.
\end{proposition}

The we will use Proposition~\ref{prop:1} to prove
Theorem~\ref{th:1}.

\section{Preliminaries and notations}

We will use the following facts in the proof of the main result of
this work.

\begin{theorem}\cite[p. 53]{Ulam-1964-b}\label{th:2} %
A unimodal map $g$ is topologically conjugated to the tent map if
and only if the complete pre-image of $0$ under the action of $g$
is dense in $[0,\, 1]$.
\end{theorem}

Remind that the set $g^{-\infty}(a) = \bigcup\limits_{n\geq
1}g^{-n}(a)$, where $$g^{-n}(a) = \{x\in [0, 1]:\,g ^n(x)=a\},$$
is called the {\bf complete pre-image} of $a$ (with respect to the
action of $g$).

The next function will play a crucial role in our reasonings. For
every $t\in \mathbb{N}$ denote
\begin{equation}\label{eq:2.1}
\xi = \xi_t:\, x \mapsto \displaystyle{\frac{1 - (-1)^{[tx]}}{2}
+(-1)^{[tx]}\{tx\}},\end{equation} where $\{\cdot \}$ denotes the
function of the fractional part of a number and $[\cdot ]$ is the
integer part.

We are now ready to formulate the result, which will be used in
our calculations.

\begin{theorem}\label{th:3}\cite[Theorem~1]{Odesa},
\cite{Plakh-Arx-Odesa} 1. Any self semi conjugacy $\xi$ of the
tent map is $\xi_t$ of the form~\eqref{eq:2.1} for some $t\geq 1$.

\noindent 2. For every $t\in \mathbb{N}$ the $\xi_t$ of the
form~\eqref{eq:2.1} is self semi conjugacy of the tent map.
\end{theorem}

Notice that~\eqref{eq:2.1} describes a piecewise linear function
$\xi_t: [0, 1]\rightarrow [0, 1]$, whose tangents are $\pm t$,
which passes through origin, and all whose kinks belong to lines
$y =0$ and $y =1$. The graphs of $\xi_5$ and $\xi_6$ are given at
Fig.~\ref{fig:01}.

\begin{figure}[ht]
\begin{center}
\begin{picture}(140,135)
\put(0,0){\vector(1,0){140}} \put(0,0){\vector(0,1){135}}

\linethickness{0.4mm} \Vidr{0}{0}{24}{120} \VidrTo{48}{0}
\VidrTo{72}{120} \VidrTo{96}{0} \VidrTo{120}{120}
\linethickness{0.1mm}
\end{picture}\hskip 3cm
\begin{picture}(140,135)
\put(0,0){\vector(1,0){140}} \put(0,0){\vector(0,1){135}}

\linethickness{0.4mm} \Vidr{0}{0}{20}{120} \VidrTo{40}{0}
\VidrTo{60}{120} \VidrTo{80}{0} \VidrTo{100}{120} \VidrTo{120}{0}
\linethickness{0.1mm}
\end{picture}
\end{center}
\caption{} \label{fig:01}\end{figure}

\begin{theorem}\cite[Theorem~1]{UMZh-2016}\label{th:4}
Assume that piecewise linear map $g$ is topologically conjugate to
the tent map via the homeomorphism $h$. If $h$ is continuously
differentiable in a subinterval of $[0, 1]$, then it is piecewise
linear.
\end{theorem}

\begin{remark}\cite[Remark~2.1]{Chaos}\label{rem:2.1}
Recall that the graph of the $n$th iteration $g^n$ of an arbitrary
continuous function $g$ of the form~\eqref{eq:1.1} has the
following properties:

1. The graph consists of $2^n$ monotone curves.

2. Each maximal part of monotonicity of $g^n$ connects the line
$y=0$ and $y=1$.

3. If $x_1,\, x_2$ are such that $\{ g^n(x_1),\, g^n(x_2)\} =
\{0,\, 1\}$ and $g^n$ is monotone on $[x_1,\, x_2]$, then
$g^{n+1}(x_1) = g^{n+1}(x_2) = 0$ and there is $x_3\in (x_1,\,
x_2)$ such that $g^{n+1}(x_3) = 1$. Moreover in this case
$g^{n+1}$ is increasing on $[x_1,\, x_3]$ and is decreasing on
$[x_3,\, x_2]$.
\end{remark}

Since, by Remark~\ref{rem:2.1}, the set $g^{-n}(0)$ consists of
$2^{n-1}+1$ points, the notation follows.

\begin{notation}\cite[Notation~2.1]{Chaos}
For every map $g:\, [0,\, 1]\rightarrow [0,\, 1]$ of the
form~\eqref{eq:1.1} and for every $n\geq 1$ denote $\{
\mu_{n,k}(g),\, 0\leq k\leq 2^{n-1}\}$ such that
$g^n(\mu_{n,k}(g))=0$ and $\mu_{n,k}(g)<\mu_{n,k+1}(g)$ for all
$k$.
\end{notation}

\begin{remark}\cite[Remark~2.2]{Chaos}\label{rem:2.3}
Notice that $\mu_{n,k}(g) = \mu_{n+1,2k}(g)$ for all $k,\, 0\leq k
\leq 2^{n-1}$.
\end{remark}

\begin{lemma}\cite[Lema~1]{Chaos}\label{lema:2.4}
For each map $g$ of the form~\eqref{eq:1.1}, every $n\geq 2$ and
$k,\, 0\leq k\leq 2^{n-2}$, the equalities $$ g(\, \mu_{n,k}(g)\,
)=\mu_{n-1,k}(g)$$ and
$$g(\, \mu_{n,k}(g)\, )= g(\,
\mu_{n,2^{n-1}-k}(g)\, ) $$ hold.
\end{lemma}

\begin{lemma}\cite[Lemma 3]{Chaos}\label{lema:2.5}
Let $g_1$ and $g_2$ be unimodal maps, and $h:\, [0,\,
1]\rightarrow [0,\, 1]$ be the conjugacy from $g_1$ to $g_2$. Then
$$ h(\mu_{n,k}(g_1)) =\mu_{n,k}(g_2) $$ for all $n\geq
1$ and $k,\, 0\leq k\leq 2^{n-1}$.
\end{lemma}

\section{Results}

\subsection{Basic properties of commutative maps}

During this section suppose that $g$ is a unimodal map of the
form~\eqref{eq:1.1} and $\psi:\, [0, 1]\rightarrow [0,1]$ is a
continuous piecewise linear unimodal map such that~\eqref{eq:1.2}
holds.

\begin{lemma}\label{lema:3.1}
$\psi(0) \in\{0;\, x_0\}$, where $x_0$ is the unique positive
fixed point of $g$.
\end{lemma}

\begin{proof}
Plug $x_0$ into~\eqref{eq:1.2} and the necessary fact follows.
\end{proof}

\begin{lemma}\label{lema:3.2}
If $\psi(0)=x_0$, where $x_0$ is a positive fixed point of $g$,
then $\psi(1)\leq x_0$.
\end{lemma}

\begin{proof}
Plug $1$ into~\eqref{eq:1.2} and the necessary fact follows.
\end{proof}

\begin{lemma}\label{lema:3.3}
$\psi(0) =0$.
\end{lemma}

\begin{proof}
By Lemma~\ref{lema:3.1}, if $\psi(0)\neq 0$, then $\psi(0)=x_0$,
where $x_0$ is a fixed point of $g$.

Suppose that $\psi$ increase at $0$, i.e. these is $\varepsilon>0$
such that $\psi$ increase at $(0, \varepsilon)$. If $1\notin
\psi\circ (0, \varepsilon)$, then it follows from~\eqref{eq:1.2}
that $\psi$ increase at $g\circ (0, \varepsilon)$. Plug $1$
into~\eqref{eq:1.2} and obtain that $g\circ \psi(1) =x_0$. Thus,
by Lemma~\ref{lema:3.2}, there exists $x^*\in (0, 1)$ such that
$\psi(x^*) =1$ and $\psi$ increase on $(0, x^*)$. without loss of
generality suppose that $x^*$ is the first zero of $\psi$. Now,
plug $x^*$ into~\eqref{eq:1.2}, whence $\psi(x^{**}) = 0$, where
$x^{**} =g(x^*)$. Thus, there is $\widetilde{x}\in (0, x^*)$ such
that $g_l(\widetilde{x}) = x^*$. This $\widetilde{x}$
transforms~\eqref{eq:1.2} to a wrong equality, because
$\psi(g(\widetilde{x})) =1$, but $g(\psi(x))<1$ for all $x\in (0,
x^*)$ (see Figure~\ref{fig:02}a. ).

We need only to consider the case when $\psi$ decrease at $0$.
This case is analogical (see Figure~\ref{fig:02}b. ). This proves
the lemma.
\end{proof}

\begin{figure}[htbp]
\begin{minipage}[h]{0.45\linewidth}
\begin{center}
\begin{picture}(220,220)
\put(110,110){\vector(-1,0){110}} \put(110,110){\vector(1,0){110}}
\put(110,110){\vector(0,1){110}} \put(110,110){\vector(0,-1){110}}

\linethickness{0.4mm} \qbezier(110,110)(85,60)(60,10)
\qbezier(60,10)(35,60)(10,110) \qbezier(110,110)(160,135)(210,160)
 \qbezier(210,160)(160,185)(110,210)
 \linethickness{0.1mm} \put(0,117){$x$} \put(0,99){$x$}

\put(99,215){$y$} \put(115,215){$x$} \put(215,118){$y$}
\put(215,98){$y$} \put(99,0){$y$} \put(115,0){$x$}
\put(110,110){\vector(0,1){110}}

\qbezier(110,160)(95,185)(80,210) \qbezier(80,210)(80,130)(80,50)
\qbezier(80,50)(95,50)(110,50)

\qbezier(160,110)(185,95)(210,80) \qbezier(210,80)(160,65)(110,50)
\qbezier(210,80)(152.5,80)(95,80) \qbezier(95,80)(95,80)(95,185)
\qbezier(95,185)(115,185)(135,185) \put(140,180){$?$}
\put(113,157){$x_0$} \put(112,82){$x^*$} \put(112,43){$x^{**}$}
\put(97,112){$\widetilde{x}$} \put(68,112){$x^*$}
\put(156,113){$x_0$}
\end{picture}
\end{center} \centerline{a. $\psi$ increase at $0$}\end{minipage}
\begin{minipage}[h]{0.45\linewidth}
\begin{center}
\begin{picture}(220,220)
\put(110,110){\vector(-1,0){110}} \put(110,110){\vector(1,0){110}}
\put(110,110){\vector(0,1){110}} \put(110,110){\vector(0,-1){110}}

\linethickness{0.4mm} \qbezier(110,110)(85,60)(60,10)
\qbezier(60,10)(35,60)(10,110) \qbezier(110,110)(160,135)(210,160)
 \qbezier(210,160)(160,185)(110,210)
 \linethickness{0.1mm} \put(0,117){$x$} \put(0,99){$x$}

\put(99,215){$y$} \put(115,215){$x$} \put(215,118){$y$}
\put(215,98){$y$} \put(99,0){$y$} \put(115,0){$x$}
\put(110,110){\vector(0,1){110}}

\qbezier(110,160)(95,135)(80,110)
\qbezier(160,110)(135,95)(110,80)
\qbezier(110,80)(102.5,80)(95,80)
\qbezier(95,80)(95,107.5)(95,135)
\qbezier(95,135)(110,135)(125,135) \put(127,130){$?$}
\end{picture}
\end{center}
\centerline{b. $\psi$ decrease at $0$}\end{minipage}
\caption{Proof of Lemma~\ref{lema:3.3}} \label{fig:02}
\end{figure}

\begin{lemma}\label{lema:3.4}
$\psi(1) \in\{0; 1\}$.
\end{lemma}

\begin{proof}
Using Lemma~\ref{lema:3.3}, plug $1$ into~\eqref{eq:1.2}, and the
fact follows.
\end{proof}

Denote $\{\mu_{2,k},\, k\geq 0\}$ the complete set of elements of
$\psi^{-2}(0)$, such that $\mu_{2,k}<\mu_{2,k+1}$ for all
admissible~$k$. By (i) of Lemma~\ref{lema:3.5}, the map $\psi$ is
monotone on $(\mu_{2,k}, \mu_{2,k+1})$ for each $k$. Denote $n$
such that $\psi_{2,n}=1$. The next fact follows from
Lemmas~\ref{lema:3.3} and~\ref{lema:3.4}.

\begin{lemma}\label{lema:3.5}
(i) For every maximal interval $I$ of monotonicity of $\psi$ we
have that $\psi(I) = [0, 1]$.

(ii) For every $k\in \{0,\ldots, n\}$ we have that $g(\mu_{2,k})
=g(\mu_{2,n-k})$ and, moreover, $$\max\{\OneOp_{[0, v]}(\mu_{2,k})
+ \OneOp_{[v, 1]}(\mu_{2,n-k}); \OneOp_{[0, v]}(\mu_{2,n-k}) +
\OneOp_{[v, 1]}(\mu_{2,k})\} =2.$$
\end{lemma}

\begin{lemma}\label{lema:3.6}
If $n=4t$ for some $t\in \mathbb{N}$, then $\psi(v)=0$.

If $n=4t+2$ for some $t\in \mathbb{N}$, then $\psi(v)=0$.

If $n$ is odd, then $\psi(v)=v$.
\end{lemma}

\begin{proof} If $n=2t$ is even, then, by (ii) of Lemma~\ref{lema:3.5},
$\mu_{2,t}=v$. Moreover, if $t=2s$ is even, than
$\psi(\mu_{2,t})=0$ and, otherwise $\psi(\mu_{2,t})=1$.

If $n=2t+1$, then, by (ii) of Lemma~\ref{lema:3.5}, $g(\mu_{2,k})
=g(\mu_{2,n-k+1})$ for $k = \frac{n+1}{2}=t+1$. Thus, $v\in
(\mu_{2,t}, \mu_{2,t+1})$. Now plus $v$ into~\eqref{eq:1.2} and
obtain $\psi(g(v)) = g(\psi(v))$. Since $\psi(1)=1$ for odd $n$,
then $1 =g(\psi(v))$, whence $\psi(v)=v$.
\end{proof}

Denote $\psi_k$ the restriction of $\psi$ on $[\mu_{2,k},
\mu_{2,k+1}]$. Denote $\psi_{k,0}$ the restriction of $\psi_k$ on
$(\mu_{2,k}, \psi_k^{-1}(v))$ and $\psi_{k,1}$ the restriction of
$\psi_k$ on $(\psi_k^{-1}(v), \mu_{2,k+1})$.

\begin{lemma}\label{lema:3.7}
(i) for every $k\leq t-1$ we have that $$ g(I_{k,s}) = I_{2k+s},\,
s\in \{0; 1\}.
$$

(ii) for every $k,\, t\leq k\leq 2t-1$ we have $$ g(I_{k,s}) =
I_{2t-1-2(k-t)-s} =I_{4t-1-2k-s},\, s\in \{0; 1\}.
$$
\end{lemma}

\begin{lemma}\label{lema:3.8}
For any $k,\, 0\leq k\leq 2t-1$ the map $\psi_{2k,s}$ increase,
and the map $\psi_{2k+1,s}$ decrease for $s\in \{0; 1\}$.
Moreover,
\begin{equation}\label{eq:3.1} \psi(I_{2k,0}) = \psi(I_{2k+1,1})
=(0, v),\end{equation} and
\begin{equation}\label{eq:3.2} \psi(I_{2k,1}) = \psi(I_{2k+1,0})
=(v, 1).\end{equation}
\end{lemma}

\begin{proof}
Lemma is obvious.
\end{proof}

\begin{lemma}\label{lema:3.9}
(i) Suppose that $k$ is such that $I_{2k}\subset (0, v)$. Then the
equality~\eqref{eq:1.2} for $x\in I_{2k+p,s}$ holds if an only if
\begin{equation}\label{eq:3.3}
g_{s +p +1 -2ps}\circ \psi_{2k+p,s} = \psi_{4k+2p+s}\circ g_1,\,
p,s\in \{0; 1\}.
\end{equation}

(ii) Suppose that $k$ is such that $I_{2k}\subset (v, 1)$. Then
the equality~\eqref{eq:1.2} for $x\in I_{2k+p,s}$ holds if an only
if
\begin{equation}\label{eq:3.4}
g_{s +p +1 -2ps}\circ \psi_{2k+p,s} = \psi_{4t-1-4k-2p-s}\circ
g_2,\, p,s\in \{0; 1\}.
\end{equation}
\end{lemma}

\begin{proof}
By~\eqref{eq:3.1}, write
\begin{eqnarray}\label{eq:3.5} g\circ \psi_{2k,0} = g_1\circ \psi_{2k,0},\\
\label{eq:3.6} g\circ \psi_{2k+1,1} = g_1\circ \psi_{2k+1,1}.
\end{eqnarray} By~\eqref{eq:3.2} write
\begin{eqnarray}\label{eq:3.7}
g\circ \psi_{2k,1} = g_2\circ \psi_{2k,0},\\ %
\label{eq:3.8} g\circ \psi_{2k+1,0} = g_2\circ \psi_{2k+1,1}.
\end{eqnarray}
Generalize~\eqref{eq:3.5} and~\eqref{eq:3.7},
\begin{equation}\label{eq:3.9}
g\circ \psi_{2k,s} = g_{s+1}\circ \psi_{2k,s},\, s\in \{0; 1\}.
\end{equation} Generalize~\eqref{eq:3.6} and~\eqref{eq:3.8},
\begin{equation}\label{eq:3.10}
g\circ \psi_{2k+1,s} = g_{2-s}\circ \psi_{2k+1,s},\, s\in \{0;
1\}.
\end{equation} At last, generalize~\eqref{eq:3.9} and~\eqref{eq:3.10}
as $$ g\circ \psi_{2k+p,s} = g_{p(2-s) +(1-p)(1+s)}\circ
\psi_{2k+1,s},\, p,s\in \{0; 1\}, $$ or
\begin{equation}\label{eq:3.11} g\circ \psi_{2k+p,s} = g_{1+p+s -2ps}\circ \psi_{2k+1,s},\, p,s\in \{0; 1\},
\end{equation}
Now part~(i) of Lemma follows from~(i) of Lemma~\ref{lema:3.7},
and~\eqref{eq:3.11}.

Part~(ii) of Lemma follows from (i) and from part (ii) of
Lemma~\ref{lema:3.7}.
\end{proof}

We are now ready to prove Theorem~\ref{th-5}.

\begin{proof}[Proof of Theorem~\ref{th-5}]
For any $k,\, 0\leq k<t$ let $\widetilde{\psi}_{4k+p}$, for $p\in
\{0; 1; 2; 3\}$, be defined on $I_{4k+p}$ as
\begin{equation}\label{eq:3.12}
\begin{array}{ll}
\widetilde{\psi}_{4k} = g_1^{-1}\circ \psi_{4k},\\
\widetilde{\psi}_{4k+1} = g_2^{-1}\circ \psi_{4k+1},\\
\widetilde{\psi}_{4k+2} = g_2^{-1}\circ \psi_{4k+2},\\
\widetilde{\psi}_{4k+3} = g_1^{-1}\circ \psi_{4k+3}.
\end{array}
\end{equation}
It follows from Lemma~\ref{lema:3.8} that $\widetilde{\psi}_{4k}$
increase $I_{4k}\rightarrow (0, v)$, the map
$\widetilde{\psi}_{4k+1}$ increase $I_{4k+1}\rightarrow (1, v)$,
the map $\widetilde{\psi}_{4k+2}$ decrease $I_{4k+2}\rightarrow
(1, v)$ and, finally, the map $\widetilde{\psi}_{4k+3}$ decrease
$I_{4k+3}\rightarrow (0, v)$.

Define $\widetilde{\psi}:\, \bigcup\limits_{k}I_k\rightarrow [0,
1]$ as $\widetilde{\psi} = \widetilde{\psi}_{k}$ on $I_k$ for each
$k,\, 0\leq k\leq 2t-1$. By Lemma~\ref{lema:3.8}, the map
$\widetilde{\psi}$ can be continuously extended to the entire $[0,
1]$, whence denote by the same letter its continuation.

It follows from the construction that
\begin{equation}\label{eq:3.13} g\circ \widetilde{\psi} = \psi.
\end{equation}
Notice, that equations~\eqref{eq:3.12} can be written as
\begin{equation}\label{eq:3.14}
\widetilde{\psi}_{4k+2p+s} = g_{s+p+1-2ps}^{-1}\circ
\psi_{4k+2p+s},\ s,p\in \{0; 1\}.
\end{equation}

Thus, for any $k,\, 0\leq k\leq 2t-1$ and $s\in \{0; 1\}$ write $$
\left.\widetilde{\psi}\circ g\right|_{I_{2k+p,s}}
\stackrel{\text{by~(i) of Lema~\ref{lema:3.7}}}{=}
\left.\widetilde{\psi}_{4k+2p+s}\circ g_1 \right|_{I_{2k+p,s}}
\stackrel{\text{by~\eqref{eq:3.14}}}{=}$$
\begin{equation}\label{eq:3.15}= g_{s+p+1-2ps}^{-1}\circ
\psi_{4k+2p+s}\circ g_1.
\end{equation}

By~\eqref{eq:3.13} and~\eqref{eq:3.15}, the
equality~\eqref{eq:1.3} is equivalent to $$ \psi =
g_{s+p+1-2ps}^{-1}\circ \psi_{4k+2p+s}\circ g_1
$$ for all $k,\, 0\leq k\leq t-1$ and $s\in \{0; 1\}$. The last
follows from (i) of Lemma~\ref{lema:3.9}.

For any $k,\, t\leq k\leq 2t-1$ and $s\in \{0; 1\}$ write
\begin{equation}\label{eq:3.16}
\left.\widetilde{\psi}\circ g\right|_{I_{2k+p,s}}
\stackrel{\text{by~(ii) of Lema~\ref{lema:3.7}}}{=}
\left.\widetilde{\psi}_{4t-1-4k-2p-s}\circ g_2
\right|_{I_{2k+p,s}}.\end{equation}

If $s=0$, then write $$ \widetilde{\psi}_{4t-1-4k-2p-s} =
\widetilde{\psi}_{4t-4k+2(1-p)-4+1}\stackrel{\text{by~\eqref{eq:3.14}}}{=}$$
$$=g^{-1}_{1+(1-p)+1-2(1-p)}\circ \psi_{4t-1-4k-2p-s} =
g^{-1}_{p+1}\circ
\psi_{4t-1-4k-2p-s}=$$\begin{equation}\label{eq:3.17}
=g^{-1}_{1+s+p-2ps}\circ \psi_{4t-1-4k-2p-s}.
\end{equation}

If $s=1$, then write $$ \widetilde{\psi}_{4t-1-4k-2p-s} =
\widetilde{\psi}_{4t-4k+2(1-p)-4}\stackrel{\text{by~\eqref{eq:3.14}}}{=}$$
$$=g^{-1}_{1+(1-p)}\circ \psi_{4t-4k-2p-s} =
g^{-1}_{2-p}\circ
\psi_{4t-4k-2p-s}=$$\begin{equation}\label{eq:3.18} =
g^{-1}_{1+s+p-2ps}\circ \psi_{4t-1-4k-2p-s}.
\end{equation}
Thus, by~\eqref{eq:3.17} and~\eqref{eq:3.18},
\begin{equation}\label{eq:3.19}
\widetilde{\psi}_{4t-1-4k-2p-s} = g^{-1}_{1+s+p-2ps}\circ
\psi_{4t-1-4k-2p-s}. g^{-1}_{1+s+p-2ps}\circ \psi_{4t-1-4k-2p-s}
\end{equation}
Now, by~\eqref{eq:3.16} and~\eqref{eq:3.19},
\begin{equation}\label{eq:3.20} \left.\widetilde{\psi}\circ
g\right|_{I_{2k+p,s}} \stackrel{\text{by~(ii) of
Lema~\ref{lema:3.7}}}{=} \left.g^{-1}_{1+s+p-2ps}\circ
\psi_{4t-1-4k-2p-s}\circ g_2 \right|_{I_{2k+p,s}}.
\end{equation}
By~\eqref{eq:3.13} and~\eqref{eq:3.20}, the
equality~\eqref{eq:1.3} is equivalent to $$ \psi
=g^{-1}_{1+s+p-2ps}\circ \psi_{4t-1-4k-2p-s}\circ g_2
$$ for all $k,\, t\leq k\leq 2t-1$ and $s\in \{0; 1\}$. the last
follows from (ii) of Lemma~\ref{lema:3.9}.
\end{proof}

\subsection{The case when unimodal map is conjugated with the tent map}

Let piecewise linear unimodal map $g$, such that
$\overline{g^{-\infty}(0)} = [0, 1]$, will be fixed till the end
of this section. By Theorem~\ref{th:2}, let $h:\, [0,
1]\rightarrow [0, 1]$ be the conjugacy from $f$ to $g$.

\begin{lemma}\label{lema:3.10}
Let $a\in (0, 1)$ be the first positive kink of $g$. If $\psi\,
'(0)>g'(0)$ for a piecewise linear self-semiconjugation $\psi$ of
$g$, then $\frac{a\cdot g'(0)}{\psi\, '(0)}$ is the first positive
kink of $\psi$.
\end{lemma}

\begin{proof}
It follows from Lemma~\ref{lema:3.15} that $\psi\, '(0)>1$.

Since $\psi$ is piecewise linear, then there is $\varepsilon>0$
such that $\psi(x)=\psi\, '(0)\cdot x$ for all $x\in (0,
\varepsilon)$.

Suppose that $\varepsilon>0$ is such that $\psi\, '(0)\cdot
\varepsilon <a$. Notice that in this case $x<a$ for all $x\in (0,
\varepsilon)$, because $\psi\, '(0)>1$. Thus, for all $x \in (0,
\varepsilon)$ we have that $\psi\, '(x)=\psi\, '(0)\cdot x$,
$(g\circ \psi)(x) = g'(0)\cdot \psi\, '(0)\cdot x$ and
$g(x)=g'(0)\cdot x$. Now, by~\eqref{eq:1.2}, $$ \psi(x) = \psi\,
'(0)\cdot x
$$ for all $x\in g\circ (0,
\varepsilon).$

Thus, for every $x\in g\circ \left( 0, \frac{a}{\psi\,
'(0)}\right)$ we have that $\psi(x) = \psi\, '(0)\cdot x$. Notice
that $g\circ \left( 0, \frac{a}{\psi\, '(0)}\right) = \left( 0,
\frac{a\cdot g'(0)}{\psi\, '(0)}\right)$. Remark that
$\frac{a\cdot g'(0)}{\psi\, '(0)}< a$.

Take an arbitrary $\delta,\, 0<\delta <a\cdot (g'(0)-1)$ is such
that $g$ is linear on $(a,\, a+\delta)$. Then for every $x\in
\left( \frac{a+\delta}{\psi\, '(0)}, \frac{a\cdot g'(0)}{\psi\,
'(0)}\right)$ we have that $\psi(x)=\psi\, '(0)\cdot x$, because
$x<\frac{a\cdot g'(0)}{\psi\, '(0)}$; also $g(x) = g'(0)\cdot x$,
because $x<a$, and $(g\circ \psi)(x) = g'(a+)\cdot (\psi\,
'(0)\cdot x -a) + g(a)$.

Now, it follows from~\eqref{eq:1.2} that
\begin{equation}\label{eq:3.21} \psi(g'(0)\cdot x) =g'(a+)\cdot
(\psi\, '(0)\cdot x -a) + g(a).
\end{equation} Denote %
$u = g'(0)\cdot x$ and remark that $u\in \left(
\frac{(a+\delta)\cdot g'(0)}{\psi\, '(0)}, \frac{a\cdot
(g'(0))^2}{\psi\, '(0)}\right)$, whenever $x\in \left(
\frac{a+\delta}{\psi\, '(0)}, \frac{a\cdot g'(0)}{\psi\,
'(0)}\right)$. Rewrite~\eqref{eq:3.21} as $$ \psi(u) = g'(a+)\cdot
\left(\frac{\psi\, '(0)\cdot u}{g'(0)} -a\right) + g(a).
$$
Clearly, if $g'(0)\neq g'(a+)$, then there is $u\in \left(
\frac{(a+\delta)\cdot g'(0)}{\psi\, '(0)}, \frac{a\cdot
(g'(0))^2}{\psi\, '(0)}\right)$ such that $\psi(u) \neq \psi\,
'(0)\cdot u$.
Remark that %
if $\delta \approx 0$, then $\frac{(a+\delta)\cdot g'(0)}{\psi\,
'(0)} \approx \frac{a\cdot g'(0)}{\psi\, '(0)}$.
\end{proof}

\begin{remark}\label{rem:3.11}
For any $n\geq 1$ and $k,\, 0\leq k\leq 2^{n-1}$ we have that
$$ \mu_{n,k}(f) = \frac{k}{2^{n-1}}.
$$
\end{remark}

The next is the direct corollary from Remark~\ref{rem:3.11}.

\begin{remark}\label{rem:3.12}
Let $t\geq 1$, and $\xi_t$ be defined by~\eqref{eq:2.1}. For any
$n\geq 1$ and $k,\, k\leq \left[ \frac{2^{n-1}}{t}\right]$ we have
that
$$\xi_t(\mu_{n,k}(f))= \mu_{n,kt}(f).$$
\end{remark}

\begin{lemma}\label{lema:3.13}
For any self semi conjugacy $\psi$ of $g$ there exists a self semi
conjugacy $\xi$ of $f$ such that $$ \psi = h\circ \xi\circ h^{-1}.
$$
\end{lemma}

\begin{proof}
Notice that $$
\left\{\begin{array}{l}\psi  = \eta\circ h^{-1},\\
\eta = \psi\circ h.
\end{array}\right.$$ provides %
the correspondence between self semi conjugations $\psi$ of $g$,
and and the semi conjugations $\eta$ from $f$ to $g$ (see
Fig.~\ref{fig:03}a).

\begin{figure}[ht]
\begin{minipage}[h]{0.45\linewidth}
$$
\xymatrix{ [0,\, 1] \ar@/_3pc/@{-->}_{\psi}[dd] \ar^{g}[rr]
\ar^{h^{-1}}[d] && [0,\, 1] \ar_{h^{-1}}[d] \ar@/^3pc/@{-->}^{\psi}[dd]\\
[0,\, 1] \ar^{f}[rr] \ar_{\eta}[d] && [0,\, 1] \ar^{\eta}[d]\\
[0,\, 1] \ar^{g}[rr] && [0,\, 1] }$$ \centerline{a. }
\end{minipage}
\hfill
\begin{minipage}[h]{0.45\linewidth}
\begin{center}
$$ \xymatrix{ [0,\, 1]
\ar@/_3pc/@{-->}_{\xi}[dd] \ar^{f}[rr]
\ar_{\eta}[d] && [0,\, 1] \ar^{\eta}[d] \ar@/^3pc/@{-->}^{\xi}[dd]\\
[0,\, 1] \ar^{g}[rr] \ar^{h^{-1}}[d] && [0,\, 1] \ar_{h^{-1}}[d]\\
[0,\, 1] \ar^{f}[rr] && [0,\, 1] }$$ \centerline{b. }\end{center}
\end{minipage}
\hfill \caption{ Illustrations} \label{fig:03}
\end{figure}
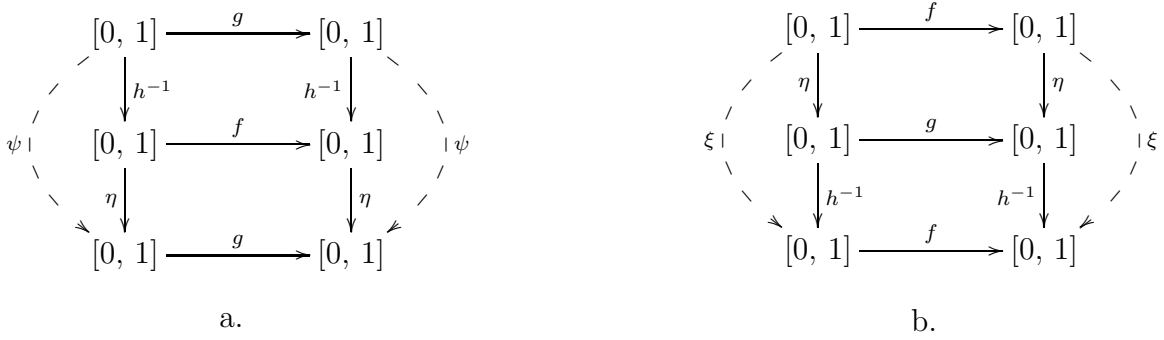

Now the correspondence
$$
\left\{\begin{array}{l}\xi  = h^{-1}\circ \eta,\\
\eta = h\circ \xi
\end{array}\right.$$ %
relates conjugacies $\eta$ from $f$ to $g$, and the self semi
conjugacies $\xi$ of the tent map (see Fig.~\ref{fig:03}b). These
relations gives the necessary equality
$$ \psi = \eta\circ h^{-1} = (h\circ \xi)\circ h^{-1}.
$$
\end{proof}

The next corollary follows from Theorem~\ref{th:3} and
Lemma~\ref{lema:3.13}.

\begin{corollary}\label{cor:3.14}
For any self semi conjugacy $\psi$ of $g$ there exists $t\in
\mathbb{N}$ such that $$\psi = h\circ \xi_t\circ h^{-1},
$$ where $\xi_t$ is determined by~\eqref{eq:2.1}.
\end{corollary}

Due to Corollary~\ref{cor:3.14}, for every $t\geq 1$ denote
\begin{equation}\label{eq:3.22} \psi_t = h\circ \xi_t\circ h^{-1}.
\end{equation}

\begin{lemma}\label{lema:3.15}
For any $t\in \mathbb{N}$ the equality
$$ \psi_t(\mu_{n,k}(g)) = \mu_{n,kt}(g)
$$ holds for all $n\geq 1$ and $k,\, k\leq \left[
\frac{2^{n-1}}{t}\right]$, where $\psi_t$ is determined
by~\eqref{eq:3.22}.
\end{lemma}

\begin{proof}
For every $n\in \mathbb{N}$ and $k,\, k\leq \left[
\frac{2^{n-1}}{t}\right]$ we have that $$ \psi_t(\mu_{n,k}(g))
\stackrel{\text{by~\eqref{eq:3.22}}}{=} (h\circ \xi_t\circ
h^{-1})(\mu_{n,k}(g)) \stackrel{\text{by Lemma~\ref{lema:2.5}}}{=}
(h\circ \xi_t)(\mu_{n,k}(f)) \stackrel{\text{by
Rem.~\ref{rem:3.12}}}{=} $$ $$=h(\mu_{n,kt}(f)) \stackrel{\text{by
Lemma~\ref{lema:2.5}}}{=} \mu_{n,kt}(g),
$$ which is necessary.
\end{proof}

\begin{lemma}\label{lema:3.16}
Let $a$ be the first kink of $g$. For every $n, k$ such that
$\mu_{n,k}(g)< a$ we have that $$ \mu_{n,2k}(g) =g'(0)\cdot
\mu_{n,k}(g).
$$
\end{lemma}

\begin{proof}
Since $g$ is linear on $(0, a)$ and $g(0)=0$, then for every $n,k$
such that $\mu_{n,k}(g)< a$ we have that $g(\mu_{n,k}(g)) =
g'(0)\cdot \mu_{n,k}(g)$. From another hand, $$ g(\mu_{n,k}(g))
\stackrel{\text{by Lemma~\ref{lema:2.4}}}{=} \mu_{n-1,k}(g)
\stackrel{\text{by Rem.~\ref{rem:2.3}}}{=} \mu_{n,2k}(g),
$$ and the lemma follows.
\end{proof}

The following lemma is similar to Lemma~\ref{lema:3.16}.

\begin{lemma}\label{lema:3.17}
Suppose that $\psi_t$ is piecewise linear for some $t\geq 1$. Let
$p$ be the first kink of $\psi_t$. For every $n, k$ such that
$\mu_{n,k}(g)< \psi(p)$ we have that $$ \mu_{n,tk}(g)
=\psi_t'(0)\cdot \mu_{n,k}(g).
$$
\end{lemma}

\begin{proof}
It follows from the linearity of $\psi_t$ at $0$, and from
Lemma~\ref{lema:3.15}.
\end{proof}

We will need the next technical fact.
\begin{lemma}\label{lema:3.18}
For every $n,k$ the set $$ \left\{\frac{k\cdot t^p}{2^{n+m}}, p\in
\mathbb{Z}_+, m\in \mathbb{Z} \right\}\cap \left[\frac{1}{2^n},
\frac{1}{2^{n-1}}\right)$$ is dense in $\left[\frac{1}{2^n},
\frac{1}{2^{n-1}}\right)$
\end{lemma}

\begin{proof}
For every $p\geq 1$ there is the unique $m\geq 1$ such that $$
\frac{k\cdot t^p}{2^{n+m}}\in \left[\frac{1}{2^n},
\frac{1}{2^{n-1}}\right).
$$ Thus, denote $m_p,\,\, t\geq 1$ be such that $$
\frac{k\cdot t^p}{2^{n+{m_p}}}\in \left[\frac{1}{2^n},
\frac{1}{2^{n-1}}\right).
$$ The latter means that $m_p,\,\, t\geq 1$ be such that $$
k\cdot t^p\in \left[ 2^{m_p}, 2\cdot 2^{m_p}\right].
$$ Apply $\log_2$ and obtain $$
p\cdot \log_2t + \log_2k \in [m_p, m_p+1],
$$ whence $$
p\mapsto p\cdot \log_2t + \log_2k-m_p
$$ will define the $p$-th iteration of the rotation by $\log_2t$
of the point $\log_2k$ on the unit circle. Since $t$ is not a
power of $2$, then $\log_2t$ is irrational, and lemma follows from
the density of any trajectory of the irrational rotation of the
unit circle.
\end{proof}

\begin{lemma}\label{lema:3.19}If $\psi_t$ is piecewise linear, then
$$
\psi_t'(0) = (g'(0))^{\log_2t}.
$$
\end{lemma}

\begin{proof}
For every $n\geq 1$ and $k\in \{0,\ldots, 2^{n-1}\}$ it follows
from Lemma~\ref{lema:3.18} that there exist sequences $(s_i),\
i\geq 1$ and $(p_i),\ i\geq 1$ such that
\begin{equation}\label{eq:3.23} \frac{k\cdot t^{p_i}}{2^{n+{s_i}}}
\rightarrow \frac{k}{2^{n-1}}
\end{equation}
for $i\rightarrow \infty$. By Remark~\ref{rem:3.11} it means that
$$ \mu_{n,k\cdot t^{p_i}\cdot 2^{-s_i}}(f) \rightarrow
\mu_{n,k}(f).
$$ By continuity of $h$ obtain $$
h(\mu_{n,k\cdot t^{p_i}\cdot 2^{-s_i}}(f)) \rightarrow
h(\mu_{n,k}(f)).
$$ and Lemma~\ref{lema:2.5}, $$
\mu_{n,k\cdot t^{p_i}\cdot 2^{-s_i}}(g) \rightarrow \mu_{n,k}(g).
$$ Now simplify $$
\mu_{n,k\cdot t^{p_i}\cdot 2^{-s_i}}(g) \stackrel{\text{by~
Lemma~\ref{lema:3.16}}}{=} (g'(0))^{-s_i}\cdot \mu_{n,k\cdot
t^{p_i}}(g) \stackrel{\text{by~ Lemma~\ref{lema:3.17}}}{=}
(g'(0))^{-s_i}\cdot (\psi_t'(0))^{p_i} \mu_{n,k}(g),
$$ whence \begin{equation}\label{eq:3.24}
(g'(0))^{-s_i}\cdot (\psi_t'(0))^{p_i} \rightarrow 1.
\end{equation}

It follows from~\eqref{eq:3.23} that
$$t^{p_i}\cdot 2^{-s_i}\rightarrow 1.
$$ Take $\log_2$ of both sides of the last limit and obtain that
$$
p_i\log_2t - s_i\rightarrow 0,
$$ whence
the sequence \begin{equation}\label{eq:3.25} \varepsilon_i =
p_i\log_2t - s_i
\end{equation} is such that
\begin{equation}\label{eq:3.26}
\varepsilon_i \rightarrow 0,\text{ when } i\rightarrow \infty.
\end{equation} %

Plug~\eqref{eq:3.25} into~\eqref{eq:3.24}, and obtain $$
(g'(0))^{\varepsilon_i -p_i\log_2t}\cdot (\psi_t'(0))^{p_i}
\rightarrow 1,
$$ whence

$$
(g'(0))^{\varepsilon_i}\cdot
\left(\frac{\psi'(0)}{(g'(0))^{\log_2t}}\right)^{p_i} \rightarrow
1.
$$ By~\eqref{eq:3.26}, the sequence $(g'(0))^{\varepsilon_i}$ is
bounded by positive numbers, whence $$
\frac{\psi'(0)}{(g'(0))^{\log_2t}} =1,
$$ and the lemma follows.
\end{proof}

We are ready now to prove Proposition~\ref{prop:1}.

\begin{proof}[Proof of Proposition~\ref{prop:1}]
It follows from Lemmas~\ref{lema:2.5} and~\ref{lema:3.16} that $$
h(\mu_{n,2k}(f)) =g'(0)\cdot h(\mu_{n,k}(f))
$$ for all $n\geq 1$ and all $k$, such that $\mu_{n,k}(g)< a$,
where $a$ is the first kink of $g$. Now, since, by
Remark~\ref{rem:3.11}, the set $\{\mu_{n,k}(f),\, n\geq 1, 0\leq
k\leq 2^{n-1}\}\cap [0, a]$ is dense in $[0, a]$, and $
\mu_{n,2k}(f) = 2\cdot \mu_{n,k}(f)$ for all $n,\, k$, then
\begin{equation}\label{eq:3.27}
h(2x) = g'(0)\cdot h(x)
\end{equation} for all $x<h^{-1}(a)$.
We can analogously use Lemma~\ref{lema:3.17} to conclude that
\begin{equation}\label{eq:3.28} h(tx) = \psi_t'(0)\cdot h(x)
\end{equation} for all $x\in [0, h^{-1}(p)]$, where $p$ is the
first kink of $\psi_t$. There there exists $x^{*}$ such
that~\eqref{eq:3.27} and~\eqref{eq:3.28} holds for all $x\in [0,
x^{*}]$.

We can find the solution of~\eqref{eq:3.27} as
\begin{equation}\label{eq:3.29} h(x) = x^{\log_2g'(0)}\cdot
\omega(\log_2x),
\end{equation} where %
$\omega: \mathbb{R}\rightarrow \mathbb{R}$ is an unknown
continuous function that that \begin{equation}\label{eq:3.30}
\omega(x+1) =\omega(x)
\end{equation} for all $x\in \mathbb{R}$.

By Lemma~\ref{lema:3.19} rewrite the functional
equation~\eqref{eq:3.28} as
$$
h(tx) = (g'(0))^{\log_2t}\cdot h(x).
$$
The latter equality and~\eqref{eq:3.29} give
$$
(tx)^{\log_2g'(0)}\cdot \omega(\log_2tx) = (g'(0))^{\log_2t}\cdot
x^{\log_2g'(0)}\cdot \omega(\log_2x).
$$
After the cancellation by $x^{\log_2g'(0)}$, we obtain
$$
t^{\log_2g'(0)}\cdot \omega(\log_2tx) = (g'(0))^{\log_2t}\cdot
\omega(\log_2x).
$$

Notice, that $(g'(0))^{\log_2t} =t^{\log_2g'(0)}$, whence,
reminding~\eqref{eq:3.30}, get
\begin{equation}\label{eq:3.31}
\left\{\begin{array}{ll}%
\omega(x+1) = \omega(1),\\
\omega(x +\log_2t) = \omega(x).
\end{array}\right.
\end{equation}

Notice that~\eqref{eq:3.31} means $x\in [0, 1]$ the values of
$\omega$ are the same on the entire trajectory of $x$ under
irrational rotation of the unit circle $[0, 1]$ by the angle
$\log_2t$. Thus, $\omega = const$ and, finally,
\begin{equation}\label{eq:3.32} h(x) = \omega \cdot x^{\log_2g'(0)}.
\end{equation}
Now we are done by Theorem~\ref{th:4}, because~\eqref{eq:3.32}
defines a continuously differentiable function.
\end{proof}

\subsection{The
case when the pre-image of $0$ under the action of the tent map is
dot dense in $[0, 1]$}

Remind that, by Theorem~\ref{th:2}, a unimodal surjective map is
topologically conjugated with the Tent map if and only if the
complete pre-image of $0$ under its action is dense in $[0, 1]$.

Thus, let $Z\subset [0, 1]$ be an open interval such that
$g^n(x)\neq 0$ for all $x\in Z$ and $n\geq 0$.

\begin{lemma}\label{lema:3.20}
Suppose that $g^{-\infty}(0)\cap Z = \emptyset$ and $\psi:\, [0,
1]\rightarrow [0, 1]$ be continuous surjective map, which commutes
with $g$. Then for any connected open interval $J\subset [0, 1]$
such that $\psi(J) = Z$ we have that $g^{-\infty}(0)\cap J =
\emptyset$.
\end{lemma}

\begin{proof}
It follows from~\eqref{eq:1.2} that $$ \psi \circ g^n = g^n\circ
\psi $$ for all $n\geq 0$. Suppose that $x\in J$ is such that
$g^n(x)=0$. That $(\psi \circ g^n)(x) = \psi(0)=0$, but $(g^n\circ
\psi)(x)\neq 0$, because $\psi(x)\in Z$.
\end{proof}

\begin{remark}\label{rem:3.21}
It follows from definitions that if $g^{-\infty}(0)\cap Z =
\emptyset$ for some open interval $Z\subset [0, 1]$, then
$g^{-\infty}(0)\cap g^{-1}(Z) = \emptyset$. \end{remark}

\begin{proof}[Proof of Theorem~\ref{th:1}]

As above, let $Z\subset [0, 1]$ be an open interval such that
$g^n(x)\neq 0$ for all $x\in Z$ and $n\geq 0$.

 Denote $\psi_l$ the maximal monotone part of $\psi$, whose
domain contains $0$. Then, by Lemma~\ref{lema:3.20} and
Remark~\ref{rem:3.21} we have that $\psi_l^{-n}(Z)\cap
g^{-\infty}(0) =\emptyset$ and $g_l^{-n}(Z)\cap g^{-\infty}(0)
=\emptyset$

Denote $x_0\in (0, 1)$ such that $g(x_0)=1$, let $x_i =
g_l(x_{i-1})$ for $i\geq 1$ and $I_i = (x_{i+1}, x_i)$ for all
$i\geq 1$.

Let $k$ be such that both $g$ and $\psi$ are linear on $(0, x_k)$.
Then $$I_{k+t} = \left(\frac{x_{k}}{(g'(0))^{t+1}},\,
\frac{x_{k}}{(g'(0))^{t}}\right),\ t\geq 0.$$

Define $\nu:\, (0, x_k)\rightarrow [0, 1)$ as $$ \nu(x)
=\frac{1}{x_k}\cdot \left\{\log_{g'(0)}x\right\},
$$ where $\{\cdot \}$ denotes a fractional part of a number.
Clearly, $$\nu\circ g_l^{-1} =\nu,$$ and $$ \nu\circ \psi_l^{-1} =
\frac{1}{x_k}\cdot \left\{\log_{g'(0)}x
-\log_{g'(0)}\psi'(0)\right\},$$ which is irrational rotation of a
unit circle, whenever $\psi$ is not an iteration of $g$.

Since any trajectory is dense in $[0, 1]$ under an irrational,
then rotation, then $[0, x_k)\cap g^{-\infty}(0) =
g_l^{-\infty}(0)$, which is a contradiction, because the bigger
pre-image of the maximum point of $g$ has also have pre-images in
$(0, x_k)$.
\end{proof}

\setlength{\unitlength}{1pt}

\pagestyle{empty}
\bibliography{Ds-Bib}{}
\bibliographystyle{makar}

\newpage
\tableofcontents

\end{document}